\documentclass{amsart}
\usepackage[utf8]{inputenc}
\usepackage{amsmath,amssymb,fullpage}
\usepackage{amsthm}
\usepackage[dvipsnames]{xcolor}
\usepackage{graphics,graphicx,verbatim}
\usepackage{float}
\usepackage{hyperref}
\usepackage{tikz}
\usetikzlibrary{patterns}

\usepackage[procnames]{listings}
\usepackage{textcomp}  % required by upquote
\usepackage{setspace}
\usepackage{inconsolata}

\lstnewenvironment{python}[1][]{
\lstset{
xleftmargin=0.0in,
language=Python,
basicstyle=\ttfamily\footnotesize\setstretch{0.75},
basewidth=0.5em,
stringstyle=\color{Purple},
showstringspaces=false,
alsoletter={1234567890},
otherkeywords={1, 2, 3, 4, 5, 6, 7, 8 ,9 , 0, -, =, +, [, ], (, ), \{, \}, :, *, !},
keywordstyle=\color{RoyalBlue},
emph={access,and,as,break,class,continue,def,del,elif,else,%
except,exec,finally,for,from,global,if,import,in,is,%
lambda,not,or,pass,raise,return,try,while,assert},
emphstyle=\color{Orange}\bfseries,
emph={[2]self},
emphstyle=[2]\color{Gray},
emph={[4]abs,all,any,apply,%
basestring,bool,buffer,callable,chr,classmethod,cmp,coerce,compile,%
complex,copyright,credits,delattr,dict,dir,divmod,enumerate,eval,%
execfile,exit,file,filter,float,frozenset,getattr,globals,hasattr,%
hash,help,hex,id,input,int,intern,isinstance,issubclass,iter,len,%
license,list,locals,long,map,max,min,object,oct,open,ord,print,pow,property,%
quit,range,raw_input,reduce,reload,repr,reversed,round,set,setattr,%
slice,sorted,staticmethod,str,sum,super,tuple,type,unichr,unicode,%
vars,xrange,zip},
emphstyle=[4]\color{Green}\bfseries,
emph={[5]sage},
emphstyle=[5]\color{RoyalBlue}\bfseries,
emph={[6]ArithmeticError,AssertionError,AttributeError,BaseException,%
DeprecationWarning,EOFError,Ellipsis,EnvironmentError,Exception,%
FloatingPointError,FutureWarning,GeneratorExit,IOError,%
ImportError,ImportWarning,IndentationError,IndexError,KeyError,%
KeyboardInterrupt,LookupError,MemoryError,NameError,None,%
NotImplemented,NotImplementedError,OSError,OverflowError,%
PendingDeprecationWarning,ReferenceError,RuntimeError,RuntimeWarning,%
StandardError,StopIteration,SyntaxError,SyntaxWarning,SystemError,%
SystemExit,TabError,TypeError,UnboundLocalError,UnicodeDecodeError,%
UnicodeEncodeError,UnicodeError,UnicodeTranslateError,UnicodeWarning,%
UserWarning,ValueError,Warning,ZeroDivisionError,TraceBack},
emphstyle=[6]\color{Red}\bfseries,
emph={[7]True},
emphstyle=[7]\color{Red}\bfseries,
emph={[8]False},
emphstyle=[8]\color{Red}\bfseries,
upquote=true,
morecomment=[s][\color{Gray}]{r"""}{"""},
commentstyle=\color{BrickRed}\slshape,
literate={|->}{{$\longrightarrow$}}{3},
procnamekeys={def,class},
procnamestyle=\color{RoyalBlue}\textbf,
framexleftmargin=1mm,
framextopmargin=1mm,
frame=single,%lines
frameround=tttt,
rulecolor=\color{ForestGreen},
backgroundcolor=\color{White}
}}{}

\newtheorem{thm}{Theorem}[section]
\newtheorem{lem}[thm]{Lemma}
\newtheorem{cor}[thm]{Corollary}
\newtheorem{prop}[thm]{Proposition}

\theoremstyle{definition}
\newtheorem{definition}[thm]{Definition}
\newtheorem{example}[thm]{Example}
\newtheorem{remark}[thm]{Remark}

\title{Homomesy on permutations with toggling actions}
\author{Will Dowling$^1$ and Nadia Lafreni\`ere$^2$}
\thanks{${}^1$Dartmouth College. For correspondence: \href{mailto:wcdowling56@gmail.com}{wcdowling56@gmail.com}. This research has been supported by the Mark C. Hansen Undergraduate Research, Scholarship, and Creativity Fund, and the Paul K. Richter and Evalyn E. Cook Richter Memorial Fund at Dartmouth College.\\
${}^2$ Concordia University. For correspondence: \href{mailto:nadia.lafreniere@concordia.ca}{nadia.lafreniere@concordia.ca}.}

\begin{document}

\maketitle

\begin{abstract}
    Homomesy is an invariance phenomenon in dynamical algebraic combinatorics which occurs when the average value of some statistic on a set of combinatorial objects is the same over each orbit generated by a map on these objects. In this paper we perform a systematic search for statistics homomesic for the set of permutations under the rotation map, identifying and proving 34 instances of homomesy. We show that these homomesies actually hold not only for rotation but in fact for a whole class of maps related to rotation by the notion of toggling, which is identified initially with composition of simple transpositions. In this way these maps are related to the rowmotion action defined on various combinatorial structures, which has a useful definition in terms of toggling. We prove some initial results on maps given by restricted or modified toggles. We discuss also the computational method used to identify candidate statistics from FindStat, a combinatorial statistics database.\\
    
\end{abstract}

%\tableofcontents

%——————————————————————————————————————————————————————————————

\section{Introduction and Methods}

Dynamical algebraic combinatorics studies how sets of combinatorial objects behave under group actions. Of particular interest is the structure of the orbits generated by such an action on such a set.

In this paper we discuss homomesy \cite{PR2015}, which occurs when some statistic on a set of combinatorial objects has the same average value over each orbit generated by an action on these objects. In particular, following the approach of \cite{ELMSW22}, we perform a systematic search for statistics homomesic for the set $S_n$ of permutations of $[n]=\{1,\dots,n\}$ under the rotation map. In Section \ref{sec2} we prove 34 instances of homomesy for permutations with rotation. In Section \ref{sec3} we give an equivalent definition of the rotation map in terms of local swaps called `toggles' and show that the homomesies of Section \ref{sec2} hold for the whole class of maps generated by these toggles. These maps are equivalent to multiplication on the right by $n$-cycles, the Coxeter elements of the symmetric group, which motivates generalization of fixed point homomesy to a result on general reflection groups. Considering rotation as toggling is interesting also because it is not uncommon for actions exhibiting homomesy to correspond to compositions of toggles \cite[Section 3.2]{Roby2016}. The toggling conception of rotation represents its close connection with rowmotion, a well-studied action on combinatorial structures such as posets and tableaux that can also be represented in terms of toggling. We conclude in Section \ref{sec4} by exploring three maps on permutations generated by restriction or modification of the toggles of Section \ref{sec3}, proving 12 more instances of homomesy.

\subsection{Permutations} Permutations are a common object of study in combinatorics and more widely in mathematics. A \textbf{permutation} $\pi$ is a bijection from a set to itself. It is convenient to restrict our study to permutations of the set $[n]=\{1,\dots,n\}$; using the normal total order on the natural numbers, we see that each of these is a unique ordering of the first $n$ natural numbers. The set of $n!$ permutations of the elements of $[n]$ is denoted by $S_n$, which along with composition comprises the symmetric group. The $i$th entry of a permutation $\pi$ is $\pi(i)$, the $i$th element in the ordering; we say equivalently that $\pi(i)$ occupies position $i$ in $\pi$. We use one-line notation, writing $\pi$ out as $\pi(1)\pi(2)\cdots\pi(n)$.

Combinatorialists often study statistics on permutations. A \textbf{permutation statistic} is a map $st:S_n \rightarrow \mathbb{Z}$ such that $st(\pi)$ evaluates some meaningful characteristic of $\pi$. Some classic permutation statistics are the number of descents, the major index, the number of excedances, and the number of inversions. An example of a map that would not be a good statistic is one that sends every permutation to the value $1$. It is not clear that this really tells us anything about a permutation.

Further, given a bijective map $\varphi: S_n \rightarrow S_n$, we say that $\varphi$ partitions $S_n$ into \textbf{orbits}, where the orbit of $\pi$, denoted $\mathcal{O}_\pi$, is the set $\{\sigma = \varphi^m(\pi)\}$. These are the permutations obtained by any number of applications of $\varphi$ to $\pi$. Thus, if $\pi$ and $\sigma$ are in distinct orbits, we can never reach $\sigma$ from $\pi$ by means of $\varphi$.
\subsection{Homomesy}
Let the triple $(X, \varphi, st)$ consist of a finite ground set $X$, a bijective map $\varphi$ from $X$ to itself, and a statistic $st$ on the elements of $X$. Then we say that the triple exhibits \textbf{homomesy} if the average value of $st$ is the same over the elements of each orbit generated by $\varphi$ on $X$. Explicitly, the triple is homomesic when
\[
\frac{1}{\#\mathcal{O}}\sum_{x\in\mathcal{O}} st(x) = c
\]
for each orbit $\mathcal{O}$ and some constant $c$. In this case we say that $(X, \varphi, st)$ is $c$-mesic. Since $X$ is for our purposes always $S_n$, and the relevant map is always clear from context, we may say unambiguously that the statistic $st$ is $c$-mesic, say, for permutations with rotation. 

\begin{remark} \label{global}
If $st$ is $c$-mesic for permutations with rotation, the global average of $st$ over $S_n$ is also $c$.
\end{remark}

This fact can be useful in finding the orbit average of a statistic we know to be homomesic. For instance, it is well-known about the distribution of fixed points over $S_n$ that permutations have 1 fixed point on average. If we determined that the fixed point statistic was homomesic but did not know the orbit average, we could without further ado conclude that it is $1$. We get the somewhat unintuitive consequence that \textit{any} map $mp$ such that $(S_n, mp, fp)$ is homomesic must more specifically be 1-mesic; see e.g. \cite{LaCroixRoby}.\\

One more helpful tool is the following well-known result:

\begin{lem}{(Linear Combinations)}\label{lin_com}
For a given action, linear combinations of homomesic statistics are also homomesic.
\end{lem}
A proof of this lemma is given in \cite[Section 3.2]{ELMSW22}. In particular, if $f$ and $g$ are homomesic statistics for a given action with orbit averages $c$ and $d$, $af+bg$ is homomesic for the same action with orbit average $ac+bd$ for any $a,b\in \mathbb{Q}$. \\

Homomesy was introduced formally by Jim Propp and Tom Roby in 2015 \cite{PR2015}, and has generated a significant literature since then. It is particularly interesting to examine naturally arising sets and maps, such as the following.
\subsection{Rotation Map}
Here we define the \textbf{rotation} map on $S_n$ and its notation.
\begin{definition} \label{map_def}
The rotation map, $rot:S_n \rightarrow S_n$, sends 
\[\pi=\pi(1)\pi(2)\dots\pi(n) \mapsto rot(\pi)=\pi(2)\dots\pi(n)\pi(1).
\]
\end{definition}
Then we have that $\pi(i+1)=rot(\pi)(i)$ if $i\in [n-1]$ and $\pi(1)=rot(\pi)(n)$. Notice that this is `backwards' (or `right-to-left') rotation, and that the corresponding `forwards' (or `left-to-right') rotation $\sigma(1)\dots\sigma(n)\mapsto\sigma(n)\sigma(1)\dots\sigma(n-1)$ is its inverse (thus $rot$ is a bijection). Denote the permutation given by $m$ applications of $rot$ to $\pi$ as $rot^m(\pi)$, which has $i$th entry $rot^m(\pi)(i)$. \\

We now give a small but illustrative example of homomesy and the rotation map in action.
\begin{example}\label{ex1}
Consider the triple $(S_3, rot, fp)$ where $fp$ counts the number of fixed points of a permutation. We compute the average value of $fp$ over each of the two orbits of size three generated by $rot$ on $S_3$:
    \[
\begin{tabular}{c c c|c}
    $fp(123)=3$; & $fp(231)=0$; & $fp(312)=0$ & $\frac{1}{3}(3+0+0)=1$ \\
    $fp(213)=1$; & $fp(132)=1$; & $fp(321)=1$ & $\frac{1}{3}(1+1+1)=1$ 
\end{tabular}.
    \]
    Since the average value of $fp$ over each orbit is $1$, we have that $(S_3, rot, fp)$ is $1$-mesic. We will see later (Prop. \ref{exc1}) that this result does not depend on the size of the permutations. In general, we are interested only in those instances of homomesy which hold in the general case, rather than for some specific $n$.
\end{example}

It is not a coincidence that both orbits generated by $rot$ on $S_3$ contain exactly three elements. We may characterize the orbit structure of $rot$ on the general $S_n$ by the following two lemmas.
\begin{lem}{(Orbit Structure)}\label{orb_struc} All orbits under the rotation map on $S_n$ have size $n$.
\end{lem}
\begin{proof}
    Notice that by Definition \ref{map_def}, $\pi(i+m\mod n)=rot^m(\pi)(i)$ (stipulating that $\pi(0)$ means $\pi(n)$). For the first time when $m=n$ we have $(i+m\mod n)=i$, that is $\pi(i+n\mod n)=rot^n(\pi)(i)=\pi(i)$. So $\pi$ and $rot^n(\pi)$ are the same permutation. Each permutation where $0\leq m< n$ is unique, and these begin to repeat starting with the $n$th permutation given by successive applications of $rot$.
\end{proof}
Intuitively, the action of $rot$ places every value in $[n]$ in each of the $n$ possible positions exactly once in each orbit. For instance, in Example \ref{ex1} above, the value $1$ occupies position $1$, position $2$, and position $3$ each exactly once in each of the two orbits. We can formalize this intuition as follows:
\begin{lem}{(Pair Lemma)}\label{pair_lemma}
For each pair $(i,j)\in [n]\times [n]$, $j=\pi(i)$ for exactly one permutation $\pi$ in each orbit generated by the rotation map.
\end{lem}
\begin{proof}
    We show first \textit{(a)} that $\pi(i)=j$ for at least one permutation $\pi$ in each orbit, and then \textit{(b)} that this holds for no more than one permutation.

\begin{itemize}
    \item[\textit{(a)}]  Consider $(i,j)\in [n]\times [n]$ and $\pi \in \mathcal{O}$. If $\pi(i)=j$, we are done. Otherwise, writing $p=\pi^{-1}(j)$, we have that $rot^{p-i}(\pi)(i)=j$ if $p>i$, and $rot^{n-i+p}(\pi)(i)=j$ if $p<i$.
    \item[\textit{(b)}] We have already accounted for the value occupying the $i$-th position for $n$ permutations in each orbit. Since all orbits have size exactly $n$ (Lemma \ref{orb_struc}), there is no value that occupies the $i$-th position a second time.
\end{itemize}
\end{proof}

Armed with this description of the orbits generated by $rot$, we may make the following characterization. The union of multisets $\bigcup_{m=0}^{n-1}\{\{(i,rot^m(\pi)(i))\mid 1\leq i\leq n\}\}$ contains all preimage-image pairs achieved in the orbit of $\pi$, with multiplicity. Lemma \ref{pair_lemma} tells us not only that each element in $[n]\times[n]$ appears in this union, but also that each element appears only once. We may therefore consider this multiset union unambiguously as a set; denoting it $\zeta_\pi$, we have for all orbits $\mathcal{O}_\pi$ that
\[
\zeta_\pi=\zeta = [n]\times [n].
\]
This equality will be especially useful because it removes all information about any particular orbit from the definition of the set. The orbit-averages of many of the statistics studied in this paper can be computed by evaluating some characteristic of $\zeta$ modified in some way. Now we have a means by which to show that this value is the same for all orbits.
\subsection{Methods: FindStat and Sagemath}

The methods used in this paper and outlined in this section were introduced in \cite{ELMSW22}. FindStat \cite[\url{www.findstat.org}]{FindStat} is an online database of combinatorial statistics and maps. As of 1 November 2023, it contains 400 permutation statistics, and 31 maps from $S_n$ to itself. FindStat has a built in interface with SageMath \cite{sage}: all FindStat statistics are accompanied by valid SageMath code, which can be accessed to compute values of the statistic. Hence we are able to check all 400 statistics for homomesy for $S_n$ with rotation for values of $n$ up to 6 with remarkable efficiency. This is especially helpful not only because checking even one statistic for homomesy by hand takes significant computational work, but also because it is not always intuitive which statistics exhibit homomesy. The following code generates a list of all permutation statistics on FindStat that are homomesic for $S_n$ (for a particular value of $n$) with $rot$:

\begin{figure}[H]
\begin{python}        
  sage: from sage.databases.findstat import FindStatMaps, FindStatStatistics  # Access to the FindStat methods
  ....: findstat()._allow_execution = True  # To run map and statistic codes from FindStat
  ....: F = DiscreteDynamicalSystem(Permutations(n), findmap(279))  # findmap(279) is rotation
  ....: for st in FindStatStatistics("Permutations"):
  ....:     if F.is_homomesic(st):
  ....:         print(st.id())  # Print list of candidates
\end{python}
\end{figure}
Based on the output of this code, we knew for which statistics to seek out proofs of the general case. Moreover, we know that any FindStat permutation statistic \textit{not} on this list is \textit{not} homomesic for permutations with rotations: SageMath has computed a counterexample. Our results are presented in the following sections. Throughout we refer to each statistic by its FindStat identification number.

\subsection{Acknowledgements} %I owe thanks most of all to my mentor, Nadia Lafreni\`ere, for help with this research and for making me a much more confident mathematician.
We thank Jim Propp, for suggesting the investigation into alternative toggling orders, and Tom Roby and Theo Douvropoulos, for thoughtful discussions of this work.
%——————————————————————————————————————————————————————————————

%——————————————————————————————————————————————————————————————

\section{Results: rotation} \label{sec2}
In this section we identify and prove 34 instances of homomesy for permutations with rotation. Homomesic statistics are divided into 4 groups: $i$th entry statistics; statistics related to excedances; statistics related to inversions; and miscellaneous statistics.

\subsection{$i$th entry statistics}

We begin with the simplest homomesy result on permutations with rotation. A similar result on binary words with a fixed number of 1s, rather than on permutations, is given in \cite[Example 2]{Roby2016}.

\begin{thm} \label{thm1}
The value of the $i$th entry of a permutation is $\frac{n+1}{2}$-mesic for permutations with rotation.
\end{thm}
\begin{proof}
From Lemma \ref{pair_lemma}, we have that each $j\in [n]$ occupies each position from $1,\dots,n$ exactly once in each orbit. In particular this means that each position $i$ is occupied exactly once in every orbit by each value $j\in [n]$.\\

So the average value of the $i$th entry in any orbit is the average of the values $1,\dots,n$, which is $\frac{n+1}{2}$; that is, the value of the $i$th entry of a permutation is $\frac{n+1}{2}$-mesic for the permutations under the rotation map.
\end{proof} 
Four instances of homomesy follow immediately:

\begin{cor}
The following statistics from the Findstat database are $\frac{n+1}{2}$-mesic for permutations with rotation:
\end{cor}
\begin{itemize}
    \item Stat 54: \textit{the first entry of a permutation} 
    \item Stat 740: \textit{the last entry of a permutation} 
    \item Stat 1806: \textit{the upper middle entry of a permutation} (i.e. $\pi(\lceil\frac{n+1}{2}\rceil)$)
    \item Stat 1807: \textit{the lower middle entry of a permutation} (i.e. $\pi(\lfloor\frac{n+1}{2}\rfloor)$)
\end{itemize}

%——————————————————————————————————————————————————————————————

\subsection{Statistics related to excedances}\label{exc_sec}
Excedances and related statistics account for 19 of the homomesies proven in this paper. \\

The standard \textbf{excedance} for a permutation $\pi$ is an index $i$ such that $\pi(i)>i$, and the related statistic is the number of excedances a permutation has. It is clear why this is called an excedance: the entry \textit{exceeds} its index in value. In the same way a \textbf{deficiency} is an index $i$ such that $\pi(i)<i$. We can weaken these inequalities, allowing $\pi(i)=i$, which gives a \textbf{weak excedance} or \textbf{weak deficiency}. Disallowing respectively $\pi(i)=i+1$ or $\pi(i)=i-1$ gives a \textbf{big excedance} or \textbf{big deficiency}. Each of these features represents a different statistic that counts it.\\

We call an index $i$  with $\pi(i)=i+k$ a \textbf{$k$-excedance} in $\pi$. So a \textbf{fixed point} is also a \textbf{$0$-excedance}, a \textbf{$1$-excedance} (also called a \textbf{small excedance}) is an entry $i$ such that $\pi(i)=i+1$, a \textbf{$(-2)$-excedance} is also a \textbf{$2$-deficiency}, and so on. This can also be considered cyclically, if we allow that $\pi(i)=(i+k \mod n)$ is a $k$-excedance. For instance, $\pi(n-1)=2$ is a cyclical $3$-excedance. We may combine these features in the expected ways: a \textbf{small weak excedance} is an index $i$ with $\pi(i)\in \{i,i+1\}$ (and so on). The following theorem lists the results of this section.
\\
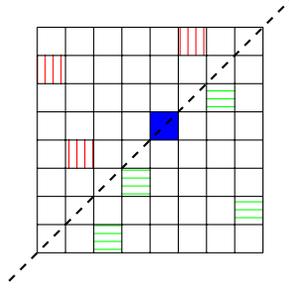
\begin{figure}
    \begin{tikzpicture}[pattern=vertical lines, pattern color=red, scale=0.375,
         baseline={([yshift=-.5ex]current bounding box.center)},
         cell71/.style={pattern=horizontal lines, pattern color=green}, cell65/.style={pattern=horizontal lines, pattern color=green},
         cell44/.style={fill=blue},
         cell57/.style={fill},
         cell32/.style={pattern=horizontal lines, pattern color=green}, cell20/.style={pattern=horizontal lines, pattern color=green},
         cell13/.style={fill}, cell06/.style={fill},
         ]
         \foreach \i in {0,...,7}
         \foreach \j in {0,...,7}
         \path[cell\i\j/.try] (\i,\j) rectangle +(1,1);
         \draw grid (8,8);
         \draw[thick,dashed] (-1,-1) -- (9,9);
\end{tikzpicture}
\caption{Excedances, Deficiencies, and Fixed Points}
\small
In this graphical representation of the permutation 74135862 of [8], squares with red vertical lines represent excedances, squares with green horizontal lines deficiencies, and the plain blue square a fixed point.
\end{figure}

\begin{thm}\label{sec3_thm}
The following statistics from the Findstat database related to excedances are homomesic for permutations with rotation:
\begin{itemize}
    \item Stat 22: \textit{the number of fixed points of a permutation} (average: 1)
    \item Stat 155: \textit{the number of excedances of a permutation} (average: $\frac{n-1}{2}$)
    \item Stat 213: \textit{the number of weak excedances of a permutation} (average: $\frac{n+1}{2}$)
    \item Stat 235: \textit{the number of indices that are not cyclical small excedances} (average: $n-1$)
    \item Stat 236: \textit{the number of cyclical small weak excedances\footnote{Permutations for which this statistic has value $0$ are sometimes called `ordinary m\'enage permutations.' See \cite{menage}.}} (average: $2$)
    \item Stat 237: \textit{the number of small excedances} (average: $\frac{n-1}{n}$)
    \item Stat 238: \textit{the number of indices that are not small weak excedances} (average: $\frac{(n-1)^2}{n}$)
    \item Stat 239: \textit{the number of small weak excedances\footnote{Permutations for which this statistic has value $0$ are sometimes called `straight m\'enage permutations.'}} (average: $1+\frac{n-1}{n}$)
    \item Stat 240: \textit{the number of indices that are not small excedances} (average: $n-\frac{n-1}{n}$)
    \item Stat 241: \textit{the number of cyclical small excedances} (average: $1$)
    \item Stat 242: \textit{the number of indices that are not cyclical small weak excedances} (average: $n-2$)
    \item Stat 648: \textit{the number of 2-excedances of a permutation} (average: $\frac{n-2}{n}$)
    \item Stat 649: \textit{the number of 3-excedances of a permutation} (average: $\frac{n-3}{n}$)
    \item Stat 673: \textit{the size of the support of a permutation (i.e. the number of non-fixed points)} (average: $n-1$)
    \item Stat 702: \textit{the number of weak deficiencies of a permutation} (average: $\frac{n+1}{2}$)
    \item Stat 703: \textit{the number of deficiencies of a permutation} (average: $\frac{n-1}{2}$)
    \item Stat 710: \textit{the number of big deficiencies of a permutation} (average: $\frac{(n-1)(n-2)}{2n}$)
    \item Stat 711: \textit{the number of big excedances of a permutation} (average: $\frac{(n-1)(n-2)}{2n}$)
    \item Stat 1439: \textit{the number of even deficiencies and of odd excedances} (average: $\lfloor \frac{n+2}{2} \rfloor$)
\end{itemize}
\end{thm}
\bigskip
We begin with those statistics counting positions whose occupants are limited to a particular value, subsequently dealing with statistics counting positions occupied by values from a specified range.
\begin{prop} \label{exc1}(Statistics 22, 235*, 236*, 237, 238*, 239*, 240*, 241, 242*, 648, 649, 673*\footnote{Statistics marked with an asterisk are linear combinations of statistics proved directly to be homomesic for permutations with rotation; we apply Lemma \ref{lin_com}.}
For any $k\in [n]$, the number of $k$-excedances is $\frac{n-k}{n}$-mesic for permutations with rotation, and $1$-mesic when considered cyclically.
\end{prop}
\begin{proof} The number of $k$-excedances in the orbit of $\pi$ is counted by the number of pairs $(i,i+k)\in \zeta_\pi$. By Lemma \ref{pair_lemma}, this is the same as the number of pairs $(i,i+k)\in[n]\times[n]$. There is one such pair for each $i\in [n-k]$. Hence each orbit has $n-k$ total $k$-excedances, and the orbit-average is $\frac{n-k}{n}$.
\\ 

Similarly, the number of cyclical $k$-excedances in the orbit of $\pi$ is counted by the number of pairs $(i,i+k\mod n)\in[n]\times[n]$. There is one such pair for each $i\in[n]$. Hence each orbit has $n$ total cyclical $k$-excedances, and the orbit average is $\frac{n}{n}=1$.
\end{proof}

\begin{prop}\label{exc} (Statistics 155, 213*, 702*, 703, 710*, 711*)
The number of excedances and the number of deficiencies of a permutation are both $\frac{n-1}{2}$-mesic for permutations with rotation.
\end{prop}
\begin{proof} We begin with excedances. The number of excedances in the orbit of $\pi$ is exactly the size of the set $\{(i,j)\in \zeta_\pi\mid j>i\}$, which by Lemma \ref{pair_lemma} is the same set as $\{(i,j)\in [n]\times[n]\mid j>i\}$ for all orbits. These all have the same size, so that excedances are homomesic for permutations with rotation with orbit-average $\frac{1}{n}\#\{(i,j)\in [n]\times[n]\mid j>i\}$. For each $i\in [n]$, there are plainly $n-i$ pairs $(i,j)$ with $j>i$. Hence the orbit-average value is $\frac{1}{n}\sum_{i\in[n]}(n-i)=\frac{1}{n}\frac{(n-1)n}{2}=\frac{n-1}{2}$. \\

A parallel argument shows that deficiencies are homomesic for permutations with rotation with orbit-average value $\frac{1}{n}\#\{(i,j)\in [n]\times[n]\mid j<i\}=\frac{n-1}{2}$. \\
\end{proof}

The statistic counting the number of even weak deficiencies and odd weak excedances merits individual treatment because of its more particular definition:

\begin{prop} (Statistic 1439)
    The number of even weak deficiencies and odd weak excedances is $\lfloor \frac{n+2}{2} \rfloor$-mesic for permutations with rotation.
\end{prop}
\begin{proof}
We count the even standard deficiencies and odd standard excedances and then add fixed points. By now the relevant line of reasoning is familiar: the number of odd excedances in the orbit of $\pi$ is the size of the set $\{(i,j)\in\zeta_\pi\mid j>i, i \text{ odd}\}$, which is the same set as $\{(i,j)\in[n]\times[n]\mid j>i, i \text{ odd}\}$, whose size is $\sum_{\text{odd }i\in[n]}(n-i)$. In the same way, the number of even deficiencies in the orbit of $\pi$ is $\sum_{\text{even }j\in[n]}(j-1)$. We compute the orbit-average of the number of even deficiencies and odd excedances when $n$ is odd as
\[
\frac{1}{n}\Bigl(\sum_{\text{odd }i\in[n]}(n-i) + \sum_{\text{even }j\in[n]}(j-1)\Bigr)= \frac{1}{n}\Bigl((n-1)+(n-3)+\dots+2)+(1+3+\dots+(n-2)\Bigr)=\frac{1}{n}\frac{n(n-1)}{2}=\frac{n-1}{2}
\]
and when $n$ is even as
\begin{gather*}
\frac{1}{n}\Bigl(\sum_{\text{odd }i\in[n]}(n-i) + \sum_{\text{even }j\in[n]}(j-1)\Bigr)=\frac{1}{n}\Bigl(((n-1)+(n-3)+\dots+1)+(1+3+\dots+(n-1))\Bigr)= \\
=\frac{1}{n}\Bigl(((n-1)+(n-3)+\dots+1)+(0+2+\dots+(n-2))+\frac{n}{2}\Bigr)=\frac{1}{n}\Bigl(\frac{n(n-1)}{2}+\frac{n}{2}\Bigr)=\frac{1}{n}\frac{n^2}{2}=\frac{n}{2}
\end{gather*}
Adding the $1$ average fixed point, we have that this statistic has orbit-average value $\frac{n+1}{2}$ when $n$ is odd and $\frac{n+2}{2}$ when $n$ is even. In both cases, this is equal to $\lfloor \frac{n+2}{2} \rfloor$.
\end{proof}

\subsection{Statistics related to inversions} This section contains 5 new instances of homomesy with more subtly related statistics.
\\

An \textit{inversion} in a permutation $\pi$ is a pair $i<j$ such that $\pi(j)<\pi(i)$. Two important statistics derived from the concept of an inversion are the \textbf{inversion number}, which is the number of inversions in a permutation, and the \textbf{inversion sum}, which is the sum $\sum(j-i)$ for all inversions $i<j$ in a permutation $\pi$. A \textit{non-inversion} in a permutation $\pi$ is a pair $i<j$ such that $\pi(i)<\pi(j)$, and the \textbf{non-inversion sum} is defined as expected: $\sum(j-i)$ for all non-inversions $i<j$.
\\

It is natural to think of the value $|\pi(i)-i|$ as the \textit{displacement} of an index $i$ under $\pi$. We then define the \textbf{total displacement} statistic of a permutation $\pi$ as the sum of the displacements of all indices, $\sum_{i\in [n]}|\pi(i)-i|$. If we limit the indices over which we sum only to those which are excedances, we get a new statistic called \textbf{depth}, which has equivalent formulations $\sum_{\pi(i)>i}(\pi(i)-i)=\#\{i\leq j:\pi(i)>j\}$ and whose value is half the total displacement. The \textbf{Spearman's rho} of a permutation and the identity permutation is defined very similarly: $\sum_i(\pi(i)-i)^2$.
\begin{thm}\label{sec_3_thm}
The following statistics from the Findstat database are homomesic for permutations with rotation:
\end{thm}
\begin{itemize}
    \item Stat 29: \textit{the depth of a permutation} (average: $\frac{1}{n}\binom{n+2}{3}$)
    \item Stat 55: \textit{the inversion sum of a permutation} (average: $\frac{1}{2}\binom{n+1}{3}$)
    \item Stat 341: \textit{the non-inversion sum of a permutation} (average: $\frac{1}{2}\binom{n+1}{3}$)
    \item Stat 828: \textit{the Spearman's rho of a permutation and the identity permutation} (average: $\binom{n+1}{3}$)
    \item Stat 830: \textit{the total displacement of a permutation} (average: $\frac{2}{n}\binom{n+2}{3}$)
\end{itemize}

We begin with depth and total displacement.

\begin{prop}\label{depth} (Statistics 29, 830)
The depth statistic is $(\frac{1}{n}\binom{n+2}{3})$-mesic, and the total displacement statistic $(\frac{2}{n}\binom{n+2}{3})$-mesic, for permutations with rotation.
\end{prop}
\begin{proof}
Using the first formulation of the depth statistic given above, we observe that the orbit-average value of depth in $\pi$'s orbit is
\[
\frac{1}{n}\sum_{\sigma\in\mathcal{O}_\pi}\sum_{i<\sigma(i)}(\sigma(i)-i)
\]
We have seen that there are exactly $n-k$ $k$-excedances in each orbit generated by $rot$ on $S_n$. Noting further that any excedance $\pi(i)>i$ is a $k$-excedance for some $k\in[n]$, it follows that for any orbit
\[
\frac{1}{n}\sum_{\sigma\in\mathcal{O}_\pi}\sum_{\sigma(i)>i}(\sigma(i)-i)=\frac{1}{n}\sum_{k\in[n]}(i+k-i)(n-k)=\frac{1}{n}\sum_{k\in[n]}k(n-k).
\]
This is $\frac{1}{n}$ times the $(n-1)$-st tetrahedral number, which (analogously to the triangular numbers) counts the spheres in a tetrahedron with edges of $(n-1)$ spheres and has value $\binom{(n-1)+2}{3}=\binom{n+1}{3}$ \cite[A000292]{OEIS}. Hence depth is $(\frac{1}{n}\binom{n+1}{3})$-mesic for permutations with rotation. 

A permutation's total displacement is twice its depth; see \cite{depth}.
Hence total displacement is $(\frac{2}{n}\binom{n+1}{3})$-mesic for permutations with rotation.
\end{proof}

Recall for present use Remark \ref{global}. From this we know that, if $(S_n,mp_1,st)$ and $(S_n,mp_2,st)$ are both homomesic, then their orbits must have the same average $st$-value.
 
\begin{prop}
(Statistic 828) The Spearman's rho of a permutation and the identity permutation is $\binom{n+1}{3}$-mesic for permutations with rotation.
\end{prop}
\begin{proof}
We compute the orbit-average value of this statistic for the orbit of $\pi$ as
\[
\frac{1}{n}\sum_{(i,j)\in\zeta_\pi}(j-i)^2
\]
which, by Lemma \ref{pair_lemma}, is the same as
\[
\frac{1}{n}\sum_{(i,j)\in[n]\times [n]}(j-i)^2.
\]
This depends solely on $n$, so all that remains is to find a closed formula for this value. By \cite[Theorem 5.7]{ELMSW22}) we conclude that the Spearman's rho of a permutation and the identity permutation is $\binom{n+1}{3}$-mesic for permutations with rotation.
\end{proof}

The values of Spearman's rho and inversion sum are related just as total displacement and depth, so that we can proceed easily to the final two statistics in this section.

\begin{prop} (Statistics 55, 341)
    Inversion sum and non-inversion sum are both $\frac{1}{2}\binom{n+1}{3}$-mesic for permutations with rotation.
\end{prop}
\begin{proof}
    In \cite{spearman_inv} it is proven that Spearman's rho is twice inversion sum. With Lemma \ref{lin_com}, this implies the result for inversion sum. For non-inversion sum, note that, given one orbit generated by $rot$ on $S_n$, we get another complete orbit by applying the reversal map $\mathcal{R}$ to all the permutations it contains. Denoting by $is$ and $nis$ the inversion sum and non-inversion sum statistics, we have
    \[
    \frac{1}{n}\sum_{\sigma\in\mathcal{O}_\pi}nis(\sigma)=\frac{1}{n}\sum_{\sigma\in\mathcal{O}_{\mathcal{R}(\pi)}}is(\sigma)=\frac{1}{n}\sum_{\sigma\in\mathcal{O}_\pi}is(\sigma)=\frac{1}{2}\binom{n+1}{3}.
    \]
The first equality follows from the fact that $nis(\pi)=is(\mathcal{R}(\pi))$ --- for reversal swaps inversions and non-inversions while preserving their size --- and the note above that reversing the permutations in one orbit gives another complete orbit; the other two follow from the fact that inversion sum is $\frac{1}{2}\binom{n+1}{3}$-mesic for permutations with rotation. This concludes the proof.
\end{proof}

It is interesting that the inversion sum statistic is homomesic for permutations with rotation, while the inversion number statistic is not. For this there is already a counterexample when $n=3$: the two orbits have differing average values of $\frac{4}{3}$ and $\frac{5}{3}$.
\subsection{Other statistics} While the 28 homomesic statistics in the three previous sections have related to such well-studied concepts as excedances and inversions, the remaining six are less familiar.

\begin{thm}\label{sec_4_thm}
The following statistics from the Findstat database are homomesic for permutations with rotation:
\begin{itemize}
    \item Stat 342: \textit{the cosine of a permutation} (average: $\frac{(n+1)^2n}{4}$)
    \item Stat 1285: \textit{the number of primes in the column sums of the two line notation of a permutation} 
    \item Stat 1287: \textit{the number of primes obtained by multiplying preimage and image of a permutation and subtracting one} 
    \item Stat 1288: \textit{the number of primes obtained by multiplying preimage and image of a permutation and adding one} 
    \item Stat 1293: \textit{the sum of all $\frac{1}{i+\pi(i)}$ for a permutation $\pi$ times the lcm $\ell$ of all possible values of $i+\pi(i)$ among permutations of the same length} (average: $\frac{\ell}{n}\Bigl((2n+1)H_{2n}-(2n+2)H_n\Bigr)$, where $H_m$ is the $m$-th harmonic number)
    \item Stat 1801: \textit{half the number of preimage-image pairs of different parity in a permutation} (average: $(\frac{1}{n}\lfloor \frac{n}{2} \rfloor\lceil\frac{n}{2}\rceil)$) 
\end{itemize}
\end{thm}

\begin{prop} (Statistic 342)
    The cosine of a permutation $\pi$, defined as $\sum_ii\pi(i)$, is $\frac{(n+1)^2n}{4}$-mesic for permutations with rotation.
\end{prop}
\begin{proof}
    We compute the average of this statistic over $\pi$'s orbit and apply Lemma \ref{pair_lemma}, giving an expression independent of $\pi$:
    \[
 \frac{1}{n} \sum_{(i,j)\in \zeta_{\pi}} ij=    \frac{1}{n} \sum_{(i,j)\in [n]\times[n]}ij.
    \]
 This means that the cosine of a permutation is homomesic with respect to rotation. From \cite[Proposition 5.31]{ELMSW22} we have that the orbit-average is $\frac{(n+1)^2n}{4}$.
\end{proof}
We prove first a lemma that will be used in the subsequent proofs.
\begin{lem} \label{multiset}
    The multiset $\{\{i+j\mid(i,j)\in[n]\times[n]\}\}$ contains each element $k\in\{2,3,\dots,2n\}$ with multiplicity $\min\{k-1,2n-k+1\}$.
\end{lem}
\begin{proof}
    The multiplicity of $k\in\{\{i+j\mid(i,j)\in[n]\times[n]\}\}$ counts the number of integer compositions of $k$ into two parts of size at most $n$. For $k\leq n+1$, we may write $k$ as a sum of $k$ $1$'s and split at any one of $k-1$ plus signs. For $k\geq n+1$, we cannot split at the first or last $k-n-1$ plus signs (which would require one part to have more than $n$ $1$'s), leaving us with $k-1-2(k-n-1)=2n-k+1$ options.
\end{proof}

Three of the statistics being considered relate directly to prime numbers, and we take them together here.

\begin{prop} (Statistics 1285, 1287, 1288)
    The number of primes in the column sums of the two line notation of a permutation, the number of primes obtained by multiplying preimage and image of a permutation and subtracting one, and the number of primes obtained by multiplying preimage and image of a permutation and adding one are homomesic for permutations with rotation.
\end{prop}
\begin{proof}
    In order, the orbit-average values of these statistics over $\pi$'s orbit are:
    \begin{gather*}
        \frac{1}{n}\#\{(i,j)\in\zeta_\pi \mid i+j\text{ is prime}\} \\
        \frac{1}{n}\#\{(i,j)\in\zeta_\pi \mid ij-1\text{ is prime}\} \\
        \frac{1}{n}\#\{(i,j)\in\zeta_\pi \mid ij+1\text{ is prime}\}.
    \end{gather*}
    As usual we prove the homomesy by showing that these values hold for all orbits by applying Lemma \ref{pair_lemma} to remove any information about $\pi$, getting, again in order,
    \begin{gather*}
        \frac{1}{n}\#\{(i,j)\in[n]\times [n] \mid i+j\text{ is prime}\} \\
        \frac{1}{n}\#\{(i,j)\in[n]\times [n] \mid ij-1\text{ is prime}\} \\
        \frac{1}{n}\#\{(i,j)\in[n]\times [n] \mid ij+1\text{ is prime}\}.
    \end{gather*}
    \end{proof}
    We leave it as an open problem to find nice expressions for these values. In the first case, using Lemma \ref{multiset}, we have $\frac{1}{n}\#\{(i,j)\in[n]\times [n] \mid i+j\text{ is prime}\}=\frac{1}{n}\sum_{i=1}^{\Pi(2n)}\min\{p_i-1,2n-p_i+1\}$ where $\Pi$ is the prime counting function, which on input $x$ gives the number of primes less than or equal to $x$, and where $p_i$ is the $i$th prime. This is sequence A292918 in the OEIS (\cite{OEIS}), and can be equivalently formulated as $\sum_{i=1}^n\Pi(n+i)-\Pi(i)$, which sums the pairs $(i,j)\in[n]\times[n]$ with $i+j$ prime for fixed $i$ as it ranges through the values in $[n]$. The other values have similar formulations.

\begin{prop} (Statistic 1293)
    The statistic $lcm\{1,2,\dots,2n\}\sum_i\frac{1}{i+\pi(i)}$ is homomesic for permutations with rotation.
\end{prop}
\begin{proof}
    We compute the average of this statistic over $\pi$'s orbit and apply Lemma \ref{pair_lemma} to see that this statistic is homomesic for permutations with rotation (note that $\ell=lcm\{1,2,\dots,2n\}$ is fixed for a given $n$):
    \[
    \frac{1}{n}\sum_{(i,j)\in\zeta_\pi}\frac{\ell}{i+j}=\frac{\ell}{n}\sum_{(i,j)\in[n]\times [n]}\frac{1}{i+j}.
    \]
    By Lemma \ref{multiset}, this is
\[
\frac{\ell}{n}\sum_{i+j=k=2}^{2n}\frac{\min\{k-1,2n-k+1\}}{k}.
\]
Since $k-1$, which increases linearly, and $2n-k+1$, which decreases linearly, coincide at $k=n+1$, we may split the sum as
\[
\frac{\ell}{n}\Bigl(\sum_{k=2}^n\frac{k-1}{k}+\sum_{k=n+1}^{2n}\frac{2n-k+1}{k}\Bigr)=\frac{\ell}{n}\Bigl((n-1)-(H_n-1)-n+(2n+1)(H_{2n}-H_n)\Bigr)=\frac{\ell}{n}\Bigl((2n+1)H_{2n}-(2n+2)H_n\Bigr)
\]
where $H_m$ is the $m$-th harmonic number, $\sum_{i=1}^m\frac{1}{i}$.
    \end{proof}

\begin{prop} (Statistic 1801)
    The statistic counting half the number of preimage-image pairs of different parity in a permutation is $(\frac{1}{n}\lfloor \frac{n}{2} \rfloor\lceil\frac{n}{2}\rceil)$-mesic for permutations with rotation.
\end{prop}
\begin{proof}
    We compute the average of this statistic over $\pi$'s orbit and apply Lemma \ref{pair_lemma} to see that this statistic is homomesic for permutations with rotation:
    \[
    \frac{1}{2n}\#\{(i,j)\in \zeta_{\pi} \mid i+j \text{ is odd}\}=\frac{1}{2n}\#\{(i,j)\in [n]\times [n] \mid i+j \text{ is odd}\}.
    \]
To get a pair $(i,j)\in[n]\times[n]$ with $i+j$ odd, we choose one odd value (out of $\lceil \frac{n}{2}\rceil$), one even value (out of $\lfloor\frac{n}{2}\rfloor$) and order them in one of two ways. Hence the orbit-average value of this statistic is $\frac{1}{2n}2\lfloor\frac{n}{2}\rfloor\lceil \frac{n}{2}\rceil=\frac{1}{n}\lfloor \frac{n}{2} \rfloor\lceil\frac{n}{2}\rceil$, as desired.

\end{proof}

This concludes the homomesy results for permutations with rotation. To reiterate: because of the computational methods used to identify these statistics as candidates, any permutation statistic listed in FindStat but not shown to be homomesic thus far \textit{is not homomesic for permutations with rotation}, for we found explicit counterexamples. The next section generalizes the homomesy results just proven.
%----------------------------------
\section{Rotation as toggling and generalization} \label{sec3}
The rotation map with which we have worked so far can be seen as a composition of toggles, or local swaps, of the entries of a permutation. In this case the toggles are simple transpositions, that is, swaps of adjacent entries. Denote by $\tau_i(\pi)$ the simple transposition swapping $\pi(i)$ and $\pi(i+1)$. Then $rot(\pi)$ is equivalent to the composition $\tau_{n-1}(\pi)\cdots\tau_{1}(\pi)$. Taking $\pi=3176524\in S_7$ as an example, we can represent this visually (Figure \ref{fig:rot_tog}). The last step is exactly the permutation $rot(\pi)$. Applying a composition of simple transpositions to $\pi$ is equivalent to multiplying $\pi$ on the right by the permutation obtained by applying the composition to the identity permutation. Thus the toggling definition of rotation on $\pi$ is equivalent to multiplying $\pi$ on the right by the long cycle $c=(12\dots n)$, which is obtained by applying the rotation sequence of toggles to the identity permutation. We can then rewrite $\mathcal{O}_\pi$ as $\{\pi\cdot c^k\mid 0\leq k\leq n-1\}=\{\pi\cdot c^k\mid 1\leq k\leq n\}$. What can we do with this toggling definition of rotation?

First, the rowmotion action, defined on various combinatorial structures, also has a natural definition in terms of toggling. Rowmotion on posets \cite{rowmotion_toggling, SW2012}, and in particular on products of chains, creates well-defined orbits for which some classic homomesy results are known \cite{PR2015}. In \cite{ben_sergi}, Ben Adenbaum and Sergi Elizalde define rowmotion on $321$-avoiding permutations by means of correspondences with Dyck paths and order ideals of Type-A root posets (and prove some homomesies for this map). But rowmotion on permutations in general remains undefined. A natural way to approach this problem is through the notion of toggling, for rowmotion on posets and tableaux have equivalent definitions of this sort. We have seen that permutations are receptive to local swapping actions in the form of simple transpositions. Another reason to consider a toggling definition of rowmotion on permutations is that the toggling definition of rowmotion on posets can be lifted to the birational and noncommutative realms while preserving important dynamical properties (\cite{lift1, lift2}).
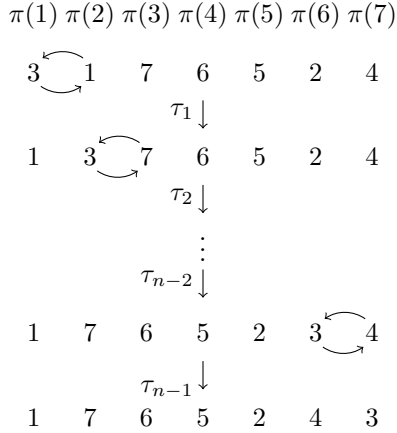
\begin{figure}
    \centering
\begin{tikzpicture}[scale=0.75]
    \node (p1) at (-3,1) {$\pi(1)$};
    \node (p2) at (-2,1) {$\pi(2)$};
    \node (p3) at (-1,1) {$\pi(3)$};
    \node (p4) at (0,1) {$\pi(4)$};
    \node (p5) at (1,1) {$\pi(5)$};
    \node (p6) at (2,1) {$\pi(6)$};
    \node (p7) at (3,1) {$\pi(7)$};

    \node (a1) at (-3,0) {3};
    \node (a2) at (-2,0) {1};
    \node (a3) at (-1,0) {7};
    \node (a4) at (0,0) {6};
    \node (a5) at (1,0) {5};
    \node (a6) at (2,0) {2};
    \node (a7) at (3,0) {4};
    %\node (at) at (-4,0) {$\tau_1(\pi):$};

    \draw[->] (-2.875,-.25) arc (235:305:.625);
    \draw[->] (-2.125,.25) arc (55:125:.625);
    \draw[->] (0,-.5) -- (0,-1) node [above left] {$\tau_1$};

    \node (b1) at (-3,-1.5) {1};
    \node (b2) at (-2,-1.5) {3};
    \node (b3) at (-1,-1.5) {7};
    \node (b4) at (0,-1.5) {6};
    \node (b5) at (1,-1.5) {5};
    \node (b6) at (2,-1.5) {2};
    \node (b7) at (3,-1.5) {4};
    %\node (bt) at (-4,-1.5) {$\tau_2(\pi):$};

    \draw[->] (-1.875,-1.75) arc (235:305:.625);
    \draw[->] (-1.125,-1.25) arc (55:125:.625);
    \draw[->] (0,-2) -- (0,-2.5) node [above left] {$\tau_2$};
    \node at (0,-3) {$\vdots$};
     \draw[->] (0,-3.5) -- (0,-4) node [above left] {$\tau_{n-2}$};

    \node (c1) at (-3,-4.625) {1};
    \node (c2) at (-2,-4.625) {7};
    \node (c3) at (-1,-4.625) {6};
    \node (c4) at (0,-4.625) {5};
    \node (c5) at (1,-4.625) {2};
    \node (c6) at (2,-4.625) {3};
    \node (c7) at (3,-4.625) {4};
    %\node (ct) at (-4.3,-4.625) {$\tau_{n-1}$};

    \draw[->] (2.125,-4.875) arc (235:305:.625);
    \draw[->] (2.875,-4.375) arc (55:125:.625);
    \draw[->] (0,-5.125) -- (0,-5.625) node [left] {$\tau_{n-1}$};

    \node (d1) at (-3,-6.125) {1};
    \node (d2) at (-2,-6.125) {7};
    \node (d3) at (-1,-6.125) {6};
    \node (d4) at (0,-6.125) {5};
    \node (d5) at (1,-6.125) {2};
    \node (d6) at (2,-6.125) {4};
    \node (d7) at (3,-6.125) {3};
    
\end{tikzpicture}
    \caption{The $rot$ map represented as a composition of toggles.}
    \label{fig:rot_tog}
\end{figure}

The question then arises (and was raised originally to the second author by Jim Propp) whether the order in which we apply these simple transpositions matters. Which of the homomesies here proven hold for the map from $S_n$ to itself given by the toggles $\{\tau_i\mid i\in[n-1]\}$ applied to $\pi$ in an arbitrary order, rather than in the specific order that produces rotation? Surprisingly, all of them do.
\begin{thm}
    All homomesies proven in Section \ref{sec2} hold not only for the rotation map, but for each map given by applying each of the simple transpositions $\{\tau_i\mid i\in[n-1]\}$ to $\pi$ exactly once in an arbitrary order.
\end{thm}
\begin{proof}
     The permutations given by applying each of the simple transpositions $\{\tau_i\mid i\in[n-1]\}$ to the identity permutation exactly once in all possible orders are exactly the $(n-1)!$ cycles of length $n$ in $S_n$.\footnote{These are called the Coxeter elements of the symmetric group; they are equivalently the conjugates of the long cycle $(12\dots n)$. In general the Coxeter elements of a reflection group are the elements given by multiplying the simple reflections each exactly once in any order.} Hence applying the simple transpositions to $\pi$ in any order is equivalent to multiplying on the right by some cycle of length $n$.
    
    We consider this right multiplication as a map from $S_n$ to itself. It is not hard to see that the orbit structure generated by right multiplication by an $n$-cycle is characterized by Lemmas \ref{orb_struc} and \ref{pair_lemma}. Clearly the map generates orbits all of size $n$, and indeed the cycle tells us precisely in which order the entries in $[n]$ occupy the positions in $[n]$. These lemmas were sufficient to prove all the indicated homomesies.
\end{proof}
This result suggests a generalization of the results of the previous section to reflection groups in general, especially since the simple transpositions $\tau_i$ have direct analogs in the simple reflections of any reflection group. We owe the following example, which generalizes the fixed point homomesy, to Theo Douvropoulos. The first step is to formulate the fixed point result in more general terms, as ``fixed point" does not have a natural interpretation for other reflection group types. For this we use some representation theory.

The number of fixed points of $\pi$ is the trace of $\pi$ under the standard representation of $S_n$. The standard representation of $S_n$ is the sum of a trivial representation, for which all elements have trace 1, and the representation indexed by $\lambda=(n-1,1)$. Thus the result that the fixed point statistic is 1-mesic for permutations with rotation is equivalent to the statement that the statistic given by the trace of $\pi$ under the representation indexed by $\lambda=(n-1,1)$ is $0$-mesic for permutations with right multiplication by any $n$-cycle $c$. That is,
\[
\sum_{k=1}^n Tr_{(n-1,1)}(\pi\cdot c^k) = 0.
\]
Since the representation indexed by $\lambda=(n-1,1)$ is the reflection representation of $S_n$, we can generalize this result in the following way. 
\begin{prop}
    Let $W$ be a reflection group with reflection representation $V$, arbitrary element $g$, and element $c$ that is the product of all simple reflections exactly once (a Coxeter element) and of order $h$. Then the statistic $Tr_V(g)$ is $0$-mesic for W with right-multiplication by $c$. That is,
    \[
    \sum_{k=1}^h Tr_V(g\cdot c^k)=0. 
    \]
\end{prop}
\begin{proof}
    Let $A$ and $B$ be the matrices of $c$ and $g$ under the reflection representation. Then $A$ is diagonalizable and $A^h=Id$. Let $\lambda_i$ be the eigenvalues of $A$; by \cite[Lemma in \S 3.16]{Humphreys1990} for all $i$, $\lambda_i\neq 1$.

    Now we have $A^*=A+A^2+\dots +A^{h-1} + Id$ is a diagonalizable matrix. It has eigenvalues $\lambda^*_i=\lambda_i+\lambda_i^2+\dots+\lambda_i^{h-1}+1=\frac{1-\lambda_i^h}{1-\lambda_i}$. Since $A^h=Id$, $\lambda_i^h=1$, and recalling that $\lambda_i\neq 1$, all these eigenvalues $\lambda^*_i$ are zero, which is to say that $A^*=\textbf{0}$, the zero matrix. Hence $B(A+A^2+\dots +A^{h-1} + Id)=BA+BA^2+\dots +BA^{h-1} + BId$ is also the zero matrix. Since the trace is a linear operator, 
    \[
    0=Tr(\textbf{0})=Tr(BA+BA^2+\dots +BA^{h-1} + BId)=\sum_{k=1}^hTr(BA^k)=\sum_{k=1}^h Tr_{V_W}(g\cdot c^k),
    \]
    as desired.
\end{proof}

\section{Restricted and modified toggles} \label{sec4}
The rich generalizations we found by considering rotation on permutations as a composition of toggles motivates further exploration by restricting or modifying the notion of toggles on permutations, which above was limited to simple transpositions. In this section we prove homomesy results for three maps given by such restriction or modification.
\subsection{Pair swapping}
Define the map `pair swapping' from $S_n$ to itself by
\[
ps(\pi)=\begin{cases}
    \pi(2)\pi(1)\pi(4)\pi(3)\dots \pi(n)\pi(n-1) & \text{if }n\text{ is even} \\
    \pi(2)\pi(1)\pi(4)\pi(3)\dots \pi(n-1)\pi(n-2)\pi(n) & \text{if }n\text{ is odd.} 
\end{cases}
\]
When represented pictorially (Figure \ref{fig:ps_tog}), we see that pair swapping is a slight restriction of the toggles which make up rotation. In particular we apply only the transpositions $\tau_i$ for odd $i$: $ps(\pi)=\tau_{n-1}(\pi)\tau_{n-3}(\pi)\cdots\tau_1(\pi)$  when $n$ is even and $\tau_{n-2}(\pi)\tau_{n-4}(\pi)\cdots\tau_1(\pi)$ when $n$ is odd. Pair swapping is clearly an involution.
\begin{figure}
    \centering
\begin{tikzpicture}[scale=0.75]
    \node (p1) at (-3,1) {$\pi(1)$};
    \node (p2) at (-2,1) {$\pi(2)$};
    \node (p3) at (-1,1) {$\pi(3)$};
    \node (p4) at (0,1) {$\pi(4)$};
    \node (p5) at (1,1) {$\pi(5)$};
    \node (p6) at (2,1) {$\pi(6)$};
    \node (p7) at (3,1) {$\pi(7)$};

    \node (a1) at (-3,0) {3};
    \node (a2) at (-2,0) {1};
    \node (a3) at (-1,0) {7};
    \node (a4) at (0,0) {6};
    \node (a5) at (1,0) {5};
    \node (a6) at (2,0) {2};
    \node (a7) at (3,0) {4};

    \draw[->] (-2.875,-.25) arc (235:305:.625);
    \draw[->] (-2.125,.25) arc (55:125:.625);
    \draw[->] (-.875,-.25) arc (235:305:.625);
    \draw[->] (-.125,.25) arc (55:125:.625);
    \draw[->] (1.125,-.25) arc (235:305:.625);
    \draw[->] (1.875,.25) arc (55:125:.625);
    \draw[->] (0,-.5) -- (0,-1);

    \node (b1) at (-3,-1.5) {1};
    \node (b2) at (-2,-1.5) {3};
    \node (b3) at (-1,-1.5) {6};
    \node (b4) at (0,-1.5) {7};
    \node (b5) at (1,-1.5) {2};
    \node (b6) at (2,-1.5) {5};
    \node (b7) at (3,-1.5) {4};
\end{tikzpicture}
    \caption{The $ps$ map represented as a composition of toggles.}
    \label{fig:ps_tog}
\end{figure}
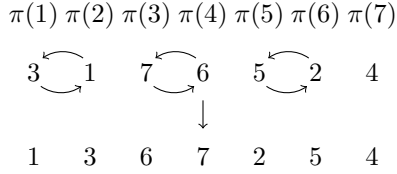

Notice that under this action each value occupies only two positions. On the other hand, the rotation map cycles each value through every position (Lemma \ref{pair_lemma})--this, we have seen, is the property of the map on which many of the homomesy results proven above depended. Pair swapping produces fewer instances of homomesy, at least for statistics in the FindStat database and closely related ones:
\begin{thm} (Statistic 1114)
    The number of odd descents\footnote{We adopt the convention that $\pi\in S_n$ has neither an ascent or a descent at position $n$. It would not be difficult to adjust these results to fit the conventions adopted by other authors.} (i.e. the number of descents at odd positions) is $\frac{1}{2}\lfloor \frac{n}{2} \rfloor$-mesic for permutations with pair swapping.
\end{thm}
\begin{proof}
    Notice that the pairs of entries which are swapped with each other are the pairs $(\pi(i),\pi(i+1))$ for odd values of $i$ less than $n$. Since a permutation $\pi\in S_n$ cannot have a descent at position $n$, we want to show that each orbit has the same number of descents at precisely these positions $i$. \\
    
    There are $\lfloor \frac{n}{2} \rfloor$ of these pairs in a permutation $\pi\in S_n$, and for each of them either $\pi(i)<\pi(i+1)$ or $\pi(i)>\pi(i+1)$. In the first case, we have $ps(\pi)(i)=\pi(i+1)>\pi(i)=ps(\pi)(i+1)$, that is, $ps(\pi)$ has a descent  at $i$ and $\pi$ does not. In the second case, $\pi$ has a descent at $i$ by definition, and $ps(\pi)$ does not, as $ps(\pi)(i)=\pi(i+1)<\pi(i)=ps(\pi)(i+1)$. So there is exactly one descent in each orbit at each of $\lfloor \frac{n}{2} \rfloor$ positions $i$ where $i$ is odd and less than $n$. As the orbits all have size 2, the average number of odd descents in each orbit is $\frac{1}{2}\lfloor \frac{n}{2} \rfloor$, as desired.
\end{proof}

    A parallel argument gives the following corollary. Indeed we observe that $\pi$ and $ps(\pi)$ have ascents at odd positions $i$ less than $n$ exactly when they do not have descents there.

    \begin{cor}
        The number of odd ascents (i.e. the number of ascents at odd positions) is $\frac{1}{2}\lfloor \frac{n}{2} \rfloor$-mesic for permutations with pair swapping.
    \end{cor}

No other statistic in the FindStat database is homomesic for permutations with pair swapping.
\subsection{Parity rotation}
Define the map `parity rotation' from $S_n$ to itself by
\[
    parrot(\pi)=\begin{cases}
    \pi(3)\pi(4)\pi(5)\dots \pi(n)\pi(1)\pi(2) & \text{if }n\text{ is even} \\
    \pi(3)\pi(4)\pi(5)\dots \pi(n)\pi(2)\pi(1) & \text{if }n\text{ is odd.} 
\end{cases}
\]
When $i>1$, this map acts on $\pi$ by moving the entry in the $i$th even (resp. odd) position to the $(i-1)$st even (resp. odd) position, and it moves the entry in the first even (resp. odd) position to the last even (resp. odd) position. Again $parrot$ is a composition of transpositions, but they are no longer simple: $\bigl(\pi(n-1),\pi(n-3)\bigr)\cdots\bigl(\pi(3),\pi(1)\bigr)\bigl(\pi(n),\pi(n-2)\bigr)\cdots\bigl(\pi(4),\pi(2)\bigr)$ when $n$ is even, and $\bigl(\pi(n),\pi(n-2)\bigr)\cdots\bigl(\pi(3),\pi(1)\bigr)\bigl(\pi(n-1),\pi(n-3)\bigr)\cdots\bigl(\pi(4),\pi(2)\bigr)$ when $n$ is odd. It is clear why the name `parity rotation' fits: we rotate the entries such that the parities of their positions remain the same. Given this action, we make the following useful definition.

There are $\lfloor \frac{n}{2} \rfloor$ entries of $\pi$ in even positions, which return to their original positions after $\lfloor \frac{n}{2} \rfloor$ applications of $parrot$; similarly the $\lceil \frac{n}{2} \rceil$ entries of $\pi$ in odd positions return to their original positions after $\lceil \frac{n}{2} \rceil$ applications of $parrot$. Thus orbits generated by $parrot$ on $S_n$ have size $lcm(\lceil \frac{n}{2} \rceil \lfloor \frac{n}{2} \rfloor)$. When $n$ is even, this is just $\frac{n}{2}$. When $n$ is odd, this is $\lceil \frac{n}{2} \rceil \lfloor \frac{n}{2} \rfloor=\frac{n^2-1}{4}$. So, for instance, the identity permutation in $S_5$ is in an orbit of size $6=\lceil\frac{5}{2}\rceil\lfloor\frac{5}{2}\rfloor$  ($12345\mapsto34521\mapsto52143\mapsto14325\mapsto32541\mapsto54123\mapsto12345$), and the identity in $S_6$ is in an orbit of size $3=\frac{6}{2}$ ($123456\mapsto345612\mapsto561234\mapsto123456$). \\

Because of the difference in orbit structure for even and odd values of $n$, there are no statistics in the FindStat database homomesic for permutations with the parity rotation map in generality. But if we restrict to even or odd values of $n$, then we find some results:

\begin{thm}
    When $n$ is even, then the following statistics are homomesic for permutations with parity rotation:
    \begin{itemize}
        \item Stat 236: \textit{the number of cyclical small weak excedances} (average: $2$)
        \item Stat 242: \textit{the number of indices that are not cyclical small weak excedances} (average: $n-2$)
    \end{itemize}
    \textit{and when $n$ is odd, then the following statistics are homomesic for permutations with parity rotation:}
    \begin{itemize}
        \item Stat 21: \textit{the number of descents of a permutation} (average: $\frac{n-1}{2}$)
        \item Stat 245: \textit{the number of ascents of a permutation} (average: $\frac{n-1}{2}$)
        \item Stat 470: \textit{the number of runs in a permutation} (average: $\frac{n+1}{2}$)
        \item Stat 325: \textit{the width of the tree associated to a permutation} (average: $\frac{n+1}{2}$)
        \item Stat 1520: \textit{the number of strict 3-descents} (average: $\frac{n-3}{2}$)
    \end{itemize}
\end{thm}

Note that when $n$ is even $parrot=rot^2$, which is not true when $n$ is odd. Hence statistics which are homomesic for permutations in $S_n$ for even $n$ with $parrot$ must indeed be homomesic for permutations with rotation, which again is not true when $n$ is odd. These facts are borne out in the theorem.

We will use the following definition in the proofs below:

\begin{definition} \label{parsets}
    Let $E_\pi=\{i=\pi(e)\mid e\text{ even}\}$, the set of entries in even positions in $\pi$, and let $D_\pi=\{j=\pi(d)\mid d\text{ odd}\}$, the set of entries in odd positions in $\pi$.
\end{definition}

Then by the definition of $parrot$, $E_\sigma=E_\pi$ and $D_\sigma=D_\pi$ for any $\sigma\in \mathcal{O}_\pi$.

\begin{prop} \label{parroteven}(Statistics 236, 242*)
    If $n$ is even, then the number of cyclical small weak excedances is 2-mesic for $S_n$ with parity rotation.
\end{prop}
We begin with a lemma that corresponds to the Pair Lemma (Lemma \ref{pair_lemma}) for the standard rotation map.
\begin{lem} \label{mod_pair}
Assume $n$ is even. Let $e\in E_\pi$, and let $d\in D_\pi$. Then $e$ occupies every even position exactly once (and never an odd position), and $d$ occupies every odd position exactly once (and never an even position), in $\mathcal{O}_\pi$, the orbit of $\pi$ generated by $parrot$.
\end{lem}
\begin{proof}[Proof of Lemma \ref{mod_pair}]
    Let $e=\pi(2i)$ for $1\leq i\leq \frac{n}{2}$, and let $2j$ be any even position. Then, interpreting $parrot^{-m}(\pi)$ as the natural inverse of $parrot$ applied to $\pi$ $m$ times, we have $parrot^{i-j}(\pi)(2j)=e$, that is, $e$ occupies even position $2j$ in the permutation $parrot^{i-j}(\pi)\in\mathcal{O}_\pi$.

    Similarly, let $d=\pi(2i-1)$ for $1\leq i\leq \frac{n}{2}$, and let $2j-1$ be any odd position. Then $parrot^{i-j}(\pi)(2j-1)=d$.

    This accounts for the position of any entry $k\in[n]$ in all $\frac{n}{2}$ permutations in the orbit of $\pi$ generated by $parrot$. The claim follows. 
\end{proof}

We are now ready to prove the proposition. Recall the definitions of $E_\pi$ and $D_\pi$ from Definition \ref{parsets}.

\begin{proof}[Proof of Proposition \ref{parroteven}]
For arbitrary $i\in[n]$ and $\pi\in S_n$, four possibilities arise:
\begin{enumerate}
    \item[\textit{(a)}] we have $i\in E_\pi$ and $i$ even;
    \item[\textit{(b)}] we have $i\in E_\pi$ and $i$ odd;
    \item[\textit{(c)}] we have $i\in D_\pi$ and $i$ even; or
    \item[\textit{(d)}] we have $i\in D_\pi$ and $i$ odd.
\end{enumerate}
Let the set $A$ contain all entries $i$ for which \textit{(a)} holds, and define the sets $B,C,D$ similarly. Clearly $A,B,C,$ and $D$ are pairwise disjoint, and $\#A+\#B+\#C+\#D=\#(A\cup B\cup C\cup D)=n$. Now, by Lemma \ref{mod_pair}, we have:
\begin{itemize}
    \item if $i\in A$, then there is a fixed point $\sigma(i)=i$ for exactly one permutation $\sigma\in\mathcal{O}_\pi$ (since $i$ is even and occupies all even positions exactly once). Further there is no $\tau\in\mathcal{O}_\pi$ with $\tau(i-1)=i$, that is, no small excedance in $\mathcal{O}_\pi$ with image $i$ (since $i-1$ is odd and $i$ occupies only even positions);
    \item if $i\in B$, then there is no fixed point at $i$ in $\mathcal{O}_\pi$ (since $i$ is odd and occupies only even positions). Further there is exactly one $\tau\in\mathcal{O}_\pi$ with $\tau(i-1)=i$ (adopting the convention that $\tau(0)=\tau(n)$) that is, exactly one small excedance in $\mathcal{O}_\pi$ with image $i$ (since $i-1$ is even, and $i$ occupies all even positions exactly once). Note that when $1\in B$, there is a \textit{cyclical} small excedance with image $1$ at $n$.
\end{itemize}
Analogously, if $i\in D$, then $i$ is the image of exactly one fixed point, and of no small excedance; and if $i\in C$, then $i$ is the image of exactly one small excedance, and of no fixed point.

To be precise, we have shown that each $i\in A\cup B\cup C\cup D=[n]$, of which there are $n$, is the \textit{image} of exactly one cyclical small weak excedance in $\mathcal{O}_\pi$. These $n$ we have counted are all the cyclical small weak excedances in the orbit; dividing by the orbit size, $\frac{n}{2}$, we obtain the desired result.
\end{proof}

When $n$ is odd, the orbit structure generated by the parity rotation map is somewhat more complicated. Lemma \ref{mod_pair} no longer holds. We have instead the following:

\begin{lem} \label{n_odd}
    For each pair $(e,d)\in E_\pi\times D_\pi$, $e$ occurs directly preceding $d$ and directly following $d$ each in exactly $\lfloor \frac{n}{2} \rfloor$ permutations in the orbit of $\pi$ generated by $parrot$.
\end{lem}

\begin{proof}
    Let $(e,d)=(\pi(i),\pi(j))\in E_\pi\times D_\pi$ (so $i$ is even and $j$ is odd). After $\lfloor \frac{n}{2} \rfloor$ applications of the map $parrot$ to $\pi$, $e$ is again in position $i$, and $d$ is in position $j+2$ if $j\neq n$, and position $1$ if $j=n$. Suppose that $j<i$. Then $d$ is in position $i-1$ for the first time after we apply $parrot$ $\lfloor \frac{n}{2} \rfloor\frac{i-j-1}{2}$ times. If $i<j$, this happens after we apply $parrot$ exactly $\lfloor \frac{n}{2} \rfloor\frac{n-j+i}{2}$ times.

    Now if we take $\frac{n-1-i}{2}$ steps `backwards' in the orbit (by applying the obvious inverse of $parrot$ this many times), we have that $e$ is in position $n-1$ and $d$ is in position $n-2$. Now and after each of the next $\frac{n-1}{2}-1=\lfloor \frac{n}{2} \rfloor-1$ applications of $parrot$, we have $e$ directly following $d$. Applying $parrot$ once more gives us a permutation with $e$ in position $n-1$ and $d$ in position $n$. Now and after each of the next $\frac{n-1}{2}-1=\lfloor \frac{n}{2} \rfloor-1$ applications of $parrot$, we have $e$ directly following $d$. Note that in this final permutation, $d$ is in the first position.

    Having seen that $e$ directly precedes (resp. follows) $d$ in at least $\lfloor \frac{n}{2} \rfloor$ permutations in each orbit, it remains only to show that they do not occur next to each other in any other permutation. To see this, recall that $parrot$ on $S_n$ generates orbits of size $\lfloor \frac{n}{2} \rfloor\lceil \frac{n}{2} \rceil$ and note that we have accounted for the entry directly preceding (resp. following) $d$ in $\lfloor \frac{n}{2} \rfloor\lfloor \frac{n}{2} \rfloor+\lfloor \frac{n}{2} \rfloor=\frac{n^2-1}{4}$ permutations in each orbit: each of $\lfloor \frac{n}{2} \rfloor$ elements of $E_\pi$ exactly $\lfloor \frac{n}{2} \rfloor$ times, and for each of these one additional permutation with $d$ in position $1$ (resp. $n$) with nothing preceding (resp. following). This accounts for every permutation in the orbit.
\end{proof}

Now we can easily show the following:

\begin{prop} \label{asc_des} (Statistics 21, 245)
    When $n$ is odd, the number of descents and the number of ascents are both $\frac{n-1}{2}$-mesic for $S_n$ with parity rotation.
\end{prop}
\begin{proof}
    There are $\lceil \frac{n}{2}\rceil \lfloor \frac{n}{2} \rfloor$ pairs $(e,d)$ in $E_\pi\times D_\pi$. For each of these pairs, either $e>d$ or $d>e$. Lemma \ref{n_odd} states that these entries appear consecutively as $ed$ and as $de$ each $\lfloor \frac{n}{2} \rfloor$ times in $\mathcal{O}_\pi$. The $\lfloor \frac{n}{2} \rfloor$ permutations when the greater entry is first have descents there, and the $\lfloor \frac{n}{2} \rfloor$ when the smaller entry is first have ascents there. Again, there are $\lfloor \frac{n}{2} \rfloor$ of each, and this is true for each of $\lceil \frac{n}{2}\rceil \lfloor \frac{n}{2} \rfloor$ pairs $(e,d)$ in $E_\pi\times D_\pi$. Dividing by orbit size, we have that the average number of descents and ascents in each orbit is $\lfloor \frac{n}{2} \rfloor$, as desired.
\end{proof}
Note that the number of runs in a permutation is one more than the number of descents -- for a run is an increasing contiguous subsequence of a permutation in one-line notation, and these are separated by descents -- and is equal to the width of the tree associated to the permutation. Thus:

\begin{cor} (Statistics 325, 470)
    When $n$ is odd, the number of runs and the width of the tree associated to the permutation are $\frac{n+1}{2}$-mesic for $S_n$ with parity rotation.
\end{cor}

And further:

\begin{prop} (Statistic 1520)
    When $n$ is odd, the number of strict 3-descents, i.e. entries $i$ such that $\pi(i)>\pi(i+3)$, is $(\frac{n-3}{2})$-mesic for $S_n$ with parity rotation.
\end{prop}
\begin{proof}
    There is a bijection between the descents in $\mathcal{O}_\pi$ and the cyclical 3-descents not at $n$. An even descent $\pi(i)>\pi(i+1)$ is mapped to the descent $parrot^{\lfloor \frac{n}{2} \rfloor}(\pi)(i)>parrot^{\lfloor \frac{n}{2} \rfloor}(\pi)(i+3\mod n)$, which is a descent since $\pi(i)=parrot^{\lfloor \frac{n}{2} \rfloor}(\pi)(i)$ and $\pi(i+1)=parrot^{\lfloor \frac{n}{2} \rfloor}(\pi)(i+3\mod n)$. An odd descent $\sigma(j)>\sigma(j+1)$ is mapped to the descent $parrot^{-\lfloor \frac{n}{2} \rfloor}(\sigma)(j)>parrot^{-\lfloor \frac{n}{2} \rfloor}(\sigma)(j+3\mod n)$, which is a descent since $\sigma(j)=parrot^{-\lfloor \frac{n}{2} \rfloor}(\sigma)(j)$ and $\sigma(j+1)=parrot^{-\lfloor \frac{n}{2} \rfloor}(\sigma)(j+3\mod n)$. Note that $parrot^{-k}(\sigma)$ denotes the permutation given by applying the inverse of $parrot$ to $\sigma$ $k$ times. We can reverse this map easily using the inverse of $parrot$, and so it is a bijection.\\

    We are not, however, counting \textit{cyclical} 3-descents in $\mathcal{O}_\pi$. We have to subtract those cyclical 3-descents that are not standard 3-descents. They correspond under the bijection just described to descents at $n-1$ and $n-2$. By Lemma \ref{n_odd}, for each of $\lceil \frac{n}{2}\rceil \lfloor \frac{n}{2} \rfloor$ pairs $(e,d)\in E_\pi\times D_\pi$, there is one such descent, at $n-2$ if $d>e$ and at $n-1$ if $e>d$. Hence we overcounted above by $\lceil \frac{n}{2}\rceil \lfloor \frac{n}{2} \rfloor$, or, dividing by orbit size, by an average of one. Since descents are $\lfloor \frac{n}{2} \rfloor$-mesic for $S_n$ with $parrot$ when $n$ is odd, 3-descents are therefore $(\lfloor \frac{n}{2} \rfloor-1)$-mesic.
\end{proof}

\subsection{Valley hopping}
Valley hopping is a toggling action on permutations with a very nice visual representation.\footnote{See \cite{valley_hopping} for a more formal definition than the one given here; note that valley hopping is called the Modified Foata-Strehl action by some authors.} We draw a permutation as a `mountain range' by placing points on the Cartesian plane at $(i,\pi(i))$ for $i\in [n]$ and drawing line segments between $(i,\pi(i))$ and $(i+1,\pi(i+1))$ for $i\in[n-1]$. Recall that a peak in a permutation is an index $i$ for which $\pi(i-1)<\pi(i)>\pi(i+1)$, and a valley is an index $j$ for which $\pi(j-1)>\pi(j)<\pi(j+1)$. The drawing process represents peaks and valleys in the permutation as pictorial peaks and valleys. Note that a peak is always a descent and a valley always an ascent; we call a descent that is not a peak a `double descent', and an ascent that is not a valley a `double ascent'. We add infinitely tall peaks before the first entry and after the last one, such that $\pi(1)$ and $\pi(n)$ can be valleys but not peaks. Then we toggle elements in the permutation by leaving peaks and valleys unchanged, and `hopping' double descents and ascents across the valleys.  The `mountain range' for the permutation  $246135\in S_6$ is shown in Figure \ref{fig:valleyhopping} with possible toggles indicated by arrows.
\begin{figure}
    \centering
  \begin{tikzpicture}[scale=0.625]

        \node[draw, circle] (linf) at (-4,7) {$\infty$};
        \node[draw, circle] (p1) at (-3,2) {2};
        \node[draw, circle] (p2) at (-2,4) {4};
        \node[draw, circle] (p3) at (-1,6) {6};
        \node[draw, circle] (p4) at (0,1) {1};
        \node[draw, circle] (p5) at (1,3) {3};
        \node[draw, circle] (p6) at (2,5) {5};
        \node[draw, circle] (rinf) at (3,7) {$\infty$};

        \draw[draw] (linf) -- (p1) -- (p2) -- (p3) -- (p4) -- (p5) -- (p6) -- (rinf);
        \draw[->] (p2) -- (-3.25,4);
        \draw[->] (p5) -- (-.25,3);
        \draw[->] (p6) -- (.75,5);
    \end{tikzpicture}
    \caption{Valley Hopping on $\pi=246135$.}
    \label{fig:valleyhopping}
\end{figure}
Denoting by $\varphi_S(\pi)$ the permutation given by hopping the elements of $S\subseteq[n]$ in this way, we have that the orbit of $\pi$ is $\{\varphi_S(\pi)\mid S\subseteq[n]\}$. Since hopping a peak or a valley does nothing, this is equivalent to $\{\varphi_t(\pi)\mid t\subseteq T\}$ where $T$ is the set of togglable elements of $\pi$, that is, the double descents and double ascents. For instance $\pi=246135\in S_6$ has togglable entries $\pi(2)=4$, $\pi(5)=3$, and $\pi(6)=5$. We toggle over all 8 subsets of these 3 elements to get that the orbit of $\pi$ is $\{246135, 426135, 246315, 246513, 426315, 426513, 246531, 426531\}$.

\begin{thm}
    The following statistics are homomesic for permutations with valley hopping:
\end{thm}
    \begin{itemize}
        \item Stat 21: \textit{the number of descents of a permutation.} (average: $\frac{n-1}{2}$)
        \item Stat 245: \textit{the number of ascents of a permutation} (average: $\frac{n-1}{2}$)
        \item Stat 325: \textit{the width of the tree associated to a permutation} (average: $\frac{n+1}{2}$)
        \item Stat 470: \textit{the number of runs in a permutation} (average: $\frac{n+1}{2}$)
    \end{itemize}
\begin{prop} (Statistic 21)
    The number of descents is $\frac{n-1}{2}$-mesic for permutations with valley hopping.
\end{prop}
\begin{proof}
    This result is a corollary (not drawn out by the authors) of \cite[Section 2]{valley_hopping}. We proceed by a much simpler and more informative counting argument. Denote the set of togglable elements in $\mathcal{O}_\pi$ by $T_{\pi}$--these are the double descents and the double ascents--and the number of peaks in $\pi$ as $pk(\pi)$. The number of valleys in $\pi$ is one greater than the number of peaks. Hence $\#T_{\pi}=n-2pk(\pi)-1$, and further $\#\mathcal{O}_{\pi}=2^{\#T_{\pi}}=2^{n-2pk(\pi)-1}$, since any subset of the elements in $T_{\pi}$ generates one permutation in the orbit of $\pi$.

    A descent in $\mathcal{O}_{\pi}$ occurs either at a peak or at some $t\in T_{\pi}$ (and never at a valley). A peak is a descent in all permutations $\sigma\in \mathcal{O}_{\pi}$, so that there are $pk(\pi)\#\mathcal{O}_{\pi}$ descents at peaks in $\mathcal{O}_{\pi}$. Further, each element $t\in T_{\pi}$ is a descent in exactly half of the permutations $\sigma\in \mathcal{O}_{\pi}$, so that there are $\#T_{\pi}\frac{\#\mathcal{O}_{\pi}}{2}$ double descents in $\mathcal{O}_{\pi}$ (double descents are not to be counted twice, despite the name; they are simply descents that are not also peaks). There are no other descents in $\mathcal{O}_{\pi}$, which therefore has a total of $pk(\pi)\#\mathcal{O}_{\pi}+\#T_{\pi}\frac{\#\mathcal{O}_{\pi}}{2}$ descents. We simplify as follows:
    \[
    pk(\pi)\#\mathcal{O}_{\pi}+\#T_{\pi}\frac{\#\mathcal{O}_{\pi}}{2}=\frac{2pk(\pi)\#\mathcal{O}_{\pi}}{2}+\frac{(n-2pk(\pi)-1)\#\mathcal{O}_{\pi}}{2}=\frac{(n-1)\#\mathcal{O}_{\pi}}{2}.
    \]
    Dividing by the size of $\pi$'s orbit allows us to retrieve the orbit-average:
    \[
    \frac{(n-1)\#\mathcal{O}_{\pi}}{2\#\mathcal{O}_{\pi}}=\frac{n-1}{2}.
    \]
    This does not depend on $\pi$, and so holds for all orbits generated by valley hopping on $S_n$, as desired.
\end{proof}
\begin{cor} (Statistics 245, 325, 470)
    The number of ascents is $\frac{n-1}{2}$-mesic, and both the number of runs of a permutation and the width of the tree associated to a permutation are $\frac{n+1}{2}$-mesic, for permutations with valley hopping.
\end{cor}
\begin{proof}
    If $des(\pi)$ is the number of descents in $\pi\in S_n$, then $\pi$ has $n-1-des(\pi)$ ascents. Hence the average number of ascents for a permutation in the orbit of $\pi$ is $n-1-\frac{n-1}{2}=\frac{n-1}{2}$. The other two results follow immediately from the fact that both of these statistics (runs and tree width) on a permutation $\pi$ are exactly one more than the number of descents of $\pi$.
\end{proof}
There is one more statistic in the FindStat database for which no counterexample to homomesy with valley hopping on permutations was computed: Statistic 461, the rix statistic (on which see \cite{valley_hopping}). Lafreni\`ere and Zhuang \cite{rix} show that the rix statistic is homomesic for permutations with valley hopping.

\bibliographystyle{plain}
\bibliography{hom_perm_tog}

\begin{thebibliography}{10}

\bibitem{ben_sergi}
Ben Adenbaum and Sergi Elizalde.
\newblock Rowmotion on 321-avoiding permutations.
\newblock {\em Electron. J. Combin.}, 30(3):Paper 3.5, 2023.

\bibitem{rowmotion_toggling}
P.~J. Cameron and D.~G. Fon-Der-Flaass.
\newblock Orbits of antichains revisited.
\newblock {\em European J. Combin.}, 16(6):545--554, 1995.

\bibitem{spearman_inv}
J.~Durbin and A.~Stuart.
\newblock Inversions and rank correlation coefficients.
\newblock {\em Journal of the Royal Statistical Society. Series B
  (Methodological)}, 13(2):303--309, 1951.

\bibitem{ELMSW22}
Jennifer Elder, Nadia Lafrenière, Erin McNicholas, Jessica Striker, and Amanda
  Welch.
\newblock Homomesies on permutations - an analysis of maps and statistics in
  the findstat database.
\newblock {\em To appear in Mathematics of Computation. ArXiv:2206.13409},
  2023.

\bibitem{lift2}
Darij Grinberg and Tom Roby.
\newblock Birational rowmotion on a rectangle over a noncommutative ring.
\newblock {\em Preprint. ArXiv:2208.11156}, 2023.

\bibitem{Humphreys1990}
James~E. Humphreys.
\newblock {\em Reflection Groups and Coxeter Groups}.
\newblock Cambridge Studies in Advanced Mathematics. Cambridge University
  Press, 1990.

\bibitem{OEIS}
OEIS~Foundation Inc.
\newblock On-line encyclopedia of number sequences.
\newblock \url{http://oeis.org/}.

\bibitem{lift1}
Michael Joseph and Tom Roby.
\newblock Birational and noncommutative lifts of antichain toggling and
  rowmotion.
\newblock {\em To appear in Algebraic Combinatorics. ArXiv:1909.09658}, 2020.

\bibitem{LaCroixRoby}
Michael La{\,}Croix and Tom Roby.
\newblock Foatic actions of the symmetric group and fixed-point homomesy.
\newblock {\em Preprint. ArXiv:2008.03292}, 2020.

\bibitem{rix}
Nadia Lafreni\`ere and Yan Zhuang.
\newblock On the rix statistic and valley-hopping.
\newblock {\em Preprint. ArXiv:2307.02711}, 2023.

\bibitem{menage}
Yiting Li.
\newblock M\'{e}nage numbers and m\'{e}nage permutations.
\newblock {\em J. Integer Seq.}, 18(6):Article, 15.6.8, 23, 2015.

\bibitem{valley_hopping}
Zhicong Lin and Jiang Zeng.
\newblock The {$\gamma$}-positivity of basic {E}ulerian polynomials via group
  actions.
\newblock {\em J. Combin. Theory Ser. A}, 135:112--129, 2015.

\bibitem{depth}
T.~Kyle Petersen and Bridget~Eileen Tenner.
\newblock The depth of a permutation.
\newblock {\em J. Comb.}, 6(1-2):145--178, 2015.

\bibitem{PR2015}
James Propp and Tom Roby.
\newblock Homomesy in products of two chains.
\newblock {\em Electron. J. Combin.}, 22(3):Paper 3.4, 29 pages, 2015.

\bibitem{Roby2016}
Tom Roby.
\newblock Dynamical algebraic combinatorics and the homomesy phenomenon.
\newblock In {\em Recent Trends in Combinatorics}, pages 619--652, Cham, 2016.

\bibitem{FindStat}
Martin Rubey, Christian Stump, et~al.
\newblock {FindStat} - {T}he combinatorial statistics database.
\newblock \url{http://www.FindStat.org}.

\bibitem{sage}
William\thinspace{}A. Stein et~al.
\newblock {\em {S}age {M}athematics {S}oftware ({V}ersion 9.8)}.
\newblock The Sage Development Team, 2023.
\newblock \url{http://www.sagemath.org}.

\bibitem{SW2012}
Jessica Striker and Nathan Williams.
\newblock Promotion and rowmotion.
\newblock {\em European J. Combin.}, 33(8):1919--1942, 2012.

\end{thebibliography}

\end{document}